\documentclass[10pt]{article}              
\usepackage{arxiv}

\usepackage{amsmath,tikz,pgfplots}
\usepackage{color}
\usepackage[english]{babel}
\usepackage[hidelinks]{hyperref}
\usepackage{amsmath,amssymb,amsthm}
\usepackage{mathtools}
\usepackage{graphicx}

\newcommand{\oR}{{\mathbb R}}

\newcommand{\oN}{{\mathbb N}}

\newcommand{\oS}{\mathbb {S}}
\newcommand{\oB}{\mathbb {B}}
\newcommand{\uf}{\underline f}
\newcommand{\of}{\overline f}
\newcommand{\ui}{\underline i}
\newcommand{\uj}{\underline j}
\newcommand{\MSymna}{\text{\rm MSym}((\oR^n)^{\otimes a})}
\newcommand{\MSymnr}{\text{\rm MSym}((\oR^n)^{\otimes r})}

\newtheorem{theorem}{Theorem}
\newtheorem{lemma}{Lemma}
\newtheorem{example}{Example}
\newtheorem{corollary}{Corollary}

\begin{document}

\title{Convergence analysis of a Lasserre hierarchy of upper bounds for polynomial  minimization on the  sphere}

\author{Etienne de Klerk \thanks{Tilburg University and Delft University of Technology, \texttt{E.deKlerk@uvt.nl}}
 \And Monique Laurent \thanks{Centrum Wiskunde \& Informatica (CWI), Amsterdam and Tilburg University, \texttt{monique@cwi.nl}}}

\maketitle

\begin{abstract}
We study the convergence rate of a  hierarchy of  upper bounds for polynomial minimization problems,
proposed by Lasserre [{\em SIAM J. Optim.} $21(3)$ $(2011)$, pp. $864-885$], for the special case when the feasible
set is the unit (hyper)sphere. The upper bound at level $r \in \mathbb{N}$ of the hierarchy is defined as the minimal expected value
of the polynomial over all probability distributions on the sphere, when the probability density function is a sum-of-squares polynomial of degree at most $2r$ with respect to the
surface measure.

We show that the exact rate of convergence is $\Theta(1/r^2)$, and explore the implications for
the related rate of convergence for the generalized problem of moments on the sphere.
\end{abstract}



\section{Introduction }\label{secintro}

We consider the problem of minimizing an $n$-variate polynomial $f:{\mathbb R}^n\to{\mathbb R}$ over a compact set $K\subseteq \oR^n$, 
i.e., the problem of  computing the parameter:

\begin{equation}\label{fmink}
f_{\min,K}:= \min_{x\in K}f(x).
\end{equation}
In this paper we will focus on the case when $K$ is the unit sphere:
$K=\oS^{n-1}=\{x\in\oR^n:\|x\|=1\}$,
in which case we will omit the subscript $K$ and simply write
$f_{\min}=\min_{x\in \oS^{n-1}} f(x).$

Problem (\ref{fmink}) is in general a computationally hard problem, already for simple sets $K$ like the hypercube, the standard simplex, and the unit ball or sphere.
For instance, the problem of finding the maximum cardinality $\alpha(G)$ of a stable set in a graph $G=([n],E)$ can be expressed as  optimizing a quadratic polynomial  over the standard simplex  \cite{MS65},
or a degree 3 polynomial over the unit sphere~\cite{Nesterov}:
\begin{eqnarray*}
  {1\over \alpha(G)} &=& \min_{x\in \oR^n} \Big\{x^T(I+A_G)x: x\ge 0, \sum_{i=1}^nx_i=1\Big\} \\&=& \min_{y \in \oS^{n-1}} \left( \sum_{i\ne j: \{i,j\} \in E} y_i^2y_j^2 + \sum_{i \in [n]} y_i^4\right), \\
  {\sqrt 2\over 3\sqrt 3} \sqrt{ 1-{1\over \alpha(G)}} &=& \max_{(y,z)\in\oS^{n+ m -1}} \sum_{ij \in\overline E} y_iy_jz_{ij}, \\
\end{eqnarray*}
where $A_G$ is the adjacency matrix of $G$, $\overline E$ is the set of non-edges of $G$ and $m=|\overline E|$.
Other applications of polynomial optimization over the unit sphere include deciding whether homogeneous polynomials
are positive semidefinite. Indeed, a homogeneous polynomial $f$ is defined as positive semidefinite precisely if
\[
f_{\min}=
\min_{x\in \oS^{n-1}} f(x)  \ge 0,
\]
and positive definite if the inequality is strict; see e.g.\ \cite{Rez00}.
As special case, one may decide if a symmetric matrix $A =  (a_{ij}) \in \mathbb{R}^{n\times n}$ is  copositive, by
deciding  if the associated form $f(x) = \sum_{i,j \in [n]} a_{ij}x_i^2x_j^2$ is positive semidefinite; see, e.g.\ \cite{Parrilo2000}.

Another special case is to decide the convexity of a homogeneous polynomial $f$, by considering the parameter
\[
\min_{(x,y)\in \oS^{2n-1}} y^T\nabla f(x)y,
\]
which is nonnegative if and only if $f$ is convex. This decision problem is known to be NP-hard, already for degree $4$ forms \cite{NPhard_Convexity_MathProg}.

\medskip
As shown by Lasserre~\cite{Las11}, the parameter (\ref{fmink}) can be reformulated via the infinite dimensional program

\begin{equation}\label{fminkreform2}
f_{\min,{K}}=\inf_{h\in\Sigma[x]}\int_{K}h(x)f(x)d\mu(x) \ \ \mbox{s.t. $\int_{{K}}h(x)d\mu(x)=1$,}
\end{equation}
where $\Sigma[x]$ denotes the set of sums of squares of polynomials, and $\mu$ is a given Borel measure supported on $K$.
\smallskip
\noindent
Given an integer $r\in \oN$,
by bounding the degree of the polynomial $h\in \Sigma[x]$ by $2r$, Lasserre \cite{Las11} defined  the parameter:

\begin{eqnarray}\label{fundr}
\of^{(r)}_K:=\min_{h\in\Sigma[x]_r}\int_{{K}}h(x)f(x)d\mu(x) \ \ \mbox{s.t. $\int_{{K}}h(x)d\mu(x)=1$,}
\end{eqnarray}
where $\Sigma[x]_r$ consists of the polynomials in $\Sigma[x]$ 
with degree at most $2r$. Here we use the `overline' symbol to indicate that the parameters provide {\em upper} bounds for $f_{\min,K}$, in contrast to the  parameters $\uf^{(r)}$ in (\ref{eqlow1}) below, which provide {\em lower} bounds for it.

Since sums of squares of polynomials can be formulated using semidefinite programming, the parameter (\ref{fundr}) can be expressed via a semidefinite program. In fact, since this program has only one affine constraint, it even admits an eigenvalue reformulation \cite{Las11}, which will be mentioned in (\ref{fminkreform3}) in  Section \ref{seccube} below.
Of course, in order to be able to compute the parameter (\ref{fundr}) in practice, one needs to know explicitly (or via some computational procedure) the moments of the reference measure $\mu$ on $K$. These moments are known for simple sets like the simplex, the box, the sphere, the ball and some simple transforms of them (they can be found,
 e.g., in Table 1 in \cite{DKL survey}).

As a direct consequence of the formulation   (\ref{fminkreform2}),  the bounds $\of^{(r)}_K$ converge asymptotically to the global minimum $f_{\min,K}$ when $r\to\infty$.
How fast the bounds converge to the global minimum in terms of the degree $r$ has been investigated  in the papers \cite{KLS MPA, DKL MOR1, DKL MOR2}, which show, respectively,  a convergence rate
in $O(1/\sqrt r)$ for general compact $K$ (satisfying a minor geometric condition), a convergence rate in $O(1/r)$ when $K$ is a convex body, and a convergence rate in $O(1/r^2)$ when $K$ is the box $[-1,1]^n$. In these works the reference measure $\mu$ is the Lebesgue measure, except for the box $[-1,1]^n$ where more general measures are considered (see Theorem \ref{theocube} below for details).

In this paper we are interested in analyzing the worst-case convergence of the bounds (\ref{fundr}) in the case of  the unit sphere $K=\oS^{n-1}$, when selecting as reference measure the surface (Haar) measure $d\sigma(x)$ on $\oS^{n-1}$. We let $\sigma_{n-1}$ denote the surface measure of  $\oS^{n-1}$, so that $d\sigma (x) /\sigma_{n-1}$ is a probability measure on $\oS^{n-1}$, with
\begin{equation}\label{eqsigma}
\sigma_{n-1}:= \int_{\oS^{n-1}} d\sigma(x)= {2\pi^{n\over 2}\over \Gamma\left({n\over 2}\right)}.
\end{equation}
(See, e.g., \cite[relation (2.2.3)]{DuX01}.)
To simplify notation we will throughout omit the subscript $K=\oS^{n-1}$ in the parameters (\ref{fmink}) and  (\ref{fundr}), which we simply denote as 
\begin{equation}\label{eqbounds}
f_{\min}=\min _{x\in \oS^{n-1}}f(x), \ \  \overline f^{(r)}= \inf_{h\in\Sigma[x]_r}\Big\{\int_{\oS^{n-1}} h(x)f(x)d\sigma(x) :  \int_{\oS^{n-1}} h(x)d\sigma(x)=1\Big\}.
\end{equation}

\begin{example}
Consider the minimization of the Motzkin form
\[
f(x_1,x_2,x_3) = x_3^6+x_1^4x_2^2+x_1^2x_2^4-3x_1^2x_2^2x_3^2
\]
on $\oS_2$. This form has $12$ minimizers on the sphere, namely $\frac{1}{\sqrt{3}}(\pm 1, \pm 1, \pm 1)$ as well as $(\pm 1,0,0)$ and $(0, \pm 1,0)$, and
one has $f_{\min} = 0$.

In Table \ref{tab:motzkin} we give the bounds $\of^{(r)}$ for the Motzkin form for $r \le 9$.
\begin{table}[h!]
  \centering
  \begin{tabular}{|c||c|c|c|c|c|c|c|c|c|c|}
    \hline
    $r$ & 0 & 1 & 2 & 3 & 4 & 5 & 6 & 7 & 8 & 9 \\
    $\of^{(r)}$ & 0.1714 & 0.0952 & 0.0519 & 0.0457 & 0.0287 & 0.0283 & 0.0193 & 0.0177 & 0.0139 &  0.0122 \\
    \hline
  \end{tabular}
  \caption{Upper bounds for the Motzkin form}\label{tab:motzkin}
\end{table}
In Figure \ref{fig:Motzkin} we show a contour plot of the Motzkin form on the sphere (top left), as well as a contour plot of the optimal density function for $r=3$ (top right), $r=6$ (bottom left), and $r=9$ (bottom right).
In the figure, the red end of the spectrum denotes higher function values.
Some local maximimizers of the Motzkin form are visible that correspond to $|x_3| = 1$ (at the poles) and $x_3 = 0$ (on the equator).

When $r = 3$ and $r=6$, the modes of the optimal density are at the global minimizers $(\pm 1,0,0)$ and $(0, \pm 1,0)$ (one may see the contours of two of
 these modes in one hemisphere).
On the other hand, when $r=9$, the mass of the distribution  is concentrated at the $8$ global minimizers $\frac{1}{\sqrt{3}}(\pm 1, \pm 1, \pm 1)$
(one may see $4$ of these in one hemisphere), and there are no modes at the global minimizers $(\pm 1,0,0)$ and $(0, \pm 1,0)$.

  \begin{figure}[h!]
  \begin{center}
 \includegraphics[width=0.3\linewidth]{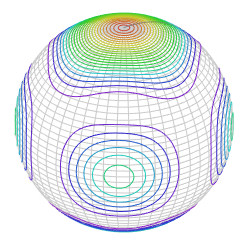} \hspace{1cm}
 \includegraphics[width=0.3\linewidth]{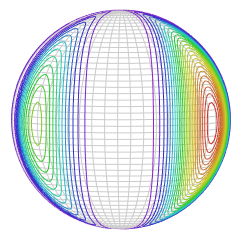}\\
 \includegraphics[width=0.3\linewidth]{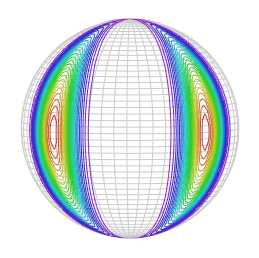}  \hspace{1cm}
   \includegraphics[width=0.3\linewidth]{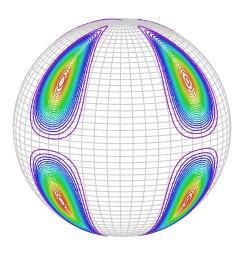}
  \caption{Contour plots of the Motzkin form on the sphere (top left) and optimal density for $r=3$ (top right), $r=6$ (bottom left), and $r=9$ (bottom right).}\label{fig:Motzkin}
\end{center}
\end{figure}

It is also illustrative to do the same plots using spherical coordinates:
\begin{eqnarray*}
x_1 &=& \sin \theta \sin \phi \\
x_2 &=& \sin \theta \cos \phi \\
x_3 &=& \cos \theta \\
\theta &\in& [0,\pi]\\
\phi &\in & [0,2\pi].
\end{eqnarray*}
In Figure \ref{fig:Motzkin2} we plot the Motzkin form in spherical coordinates (top left), as well as the optimal density function that corresponds to $r=3$ (top right), $r=6$ (bottom left), and $r=9$ (bottom right).
\begin{figure}[h!]
 \includegraphics[width=0.5\linewidth]{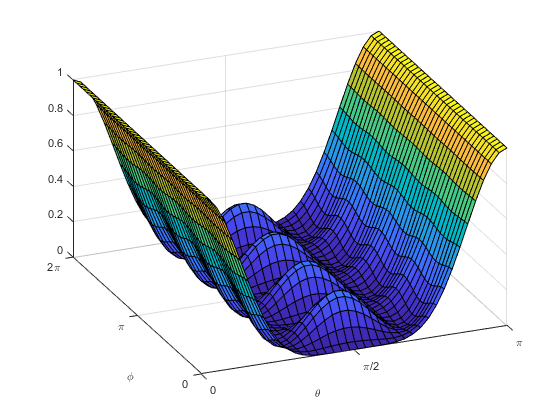}
  \includegraphics[width=0.5\linewidth]{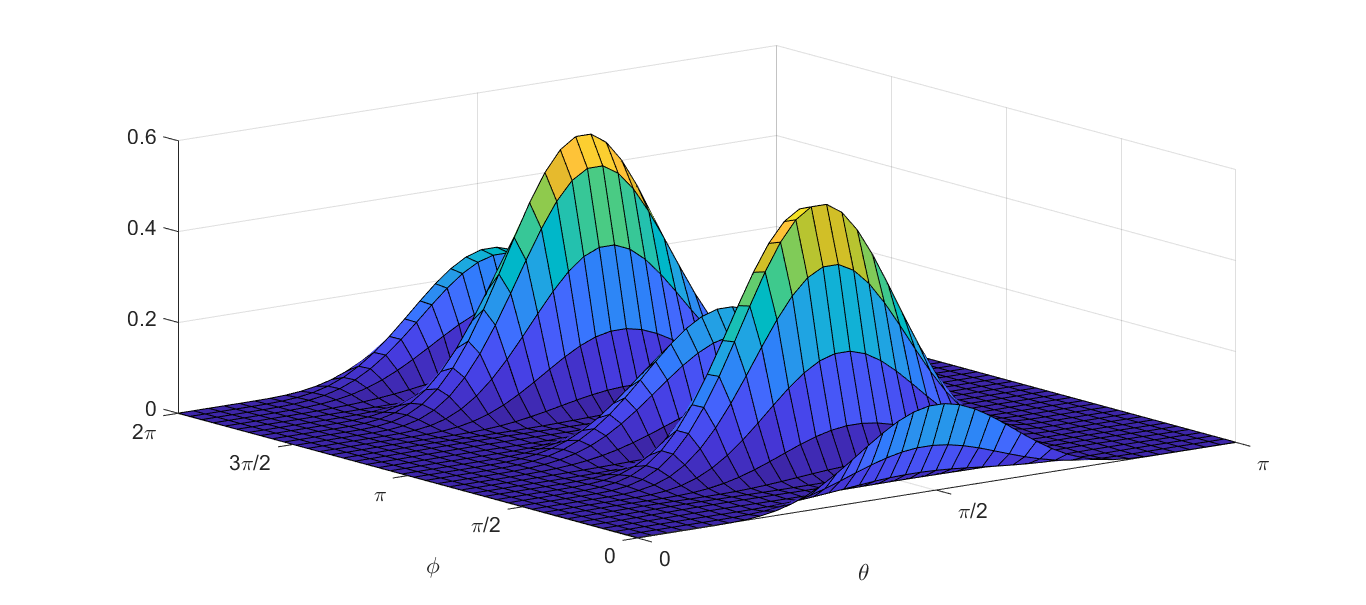}
  \includegraphics[width=0.5\linewidth]{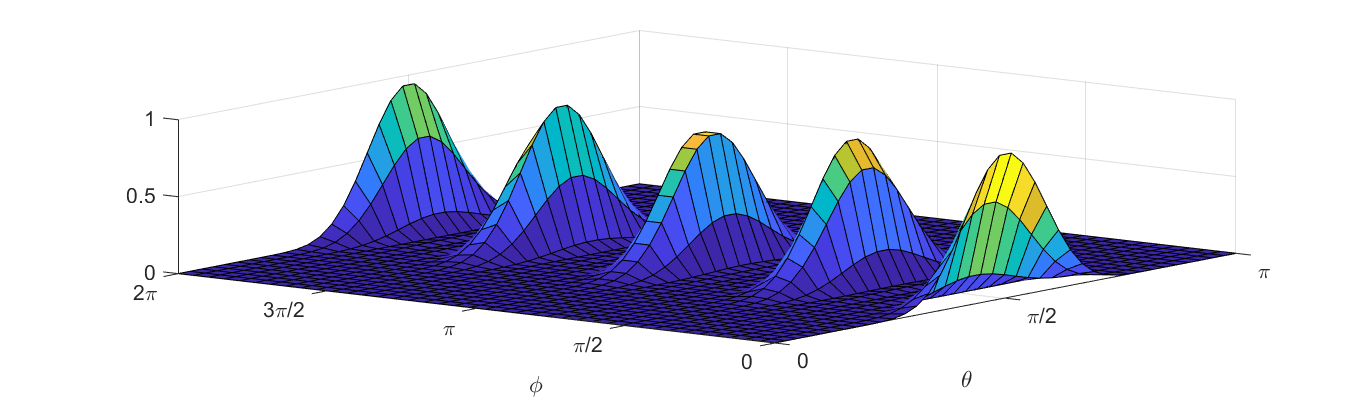}
   \includegraphics[width=0.5\linewidth]{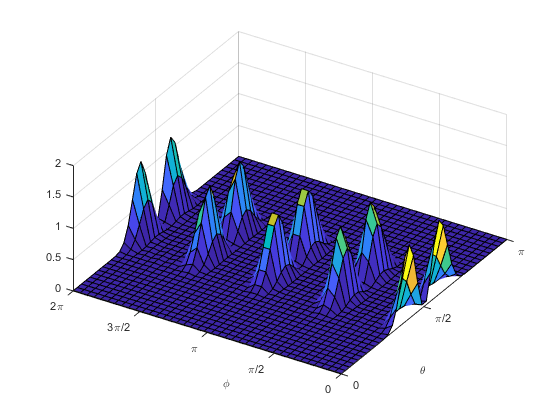}
  \caption{Plots of the Motzkin form on the sphere (top left) and optimal density for $r=3$ (top right), $r=6$ (bottom left), and $r=9$ (bottom right), in spherical coordinates.}\label{fig:Motzkin2}
\end{figure}
For example, when $r=9$ one can see the $8$ modes (peaks) of the density that correspond to the $8$ global minimizers $\frac{1}{\sqrt{3}}(\pm 1, \pm 1, \pm 1)$. (Note that the peaks at $\phi = 0$ and $\phi = 2\pi$ correspond to the same mode of the density, due to periodicity.)
Likewise when $r=3$ and $r=6$ one may see $4$ modes corresponding to $(\pm 1,0,0)$ and $(0, \pm 1,0)$.

\end{example}

The convergence rate of the bounds $\of^{(r)}$ was investigated by Doherty and Wehner \cite{DW}, who showed
\begin{equation}\label{eqDW}
\of^{(r)}-f_{\min}=O\left({1\over r}\right)
\end{equation}
when $f$ is a {\em homogeneous} polynomial. As we will briefly recap in   Section \ref{secDW},
their result  follows in fact as a byproduct of their analysis of another Lasserre hierarchy of  bounds for $f_{\min}$,  namely the {\em lower} bounds
(\ref{eqlow1}) below.

Our main contribution in this paper is to show that the convergence rate of the bounds $\of^{(r)}$  is  $O(1/r^2)$ for {\em any} polynomial $f$ and, moreover, that this analysis is tight for any (nonzero) linear polynomial $f$.
This is summarized in the following theorem.

\begin{theorem}\label{theomain}
\begin{itemize}
\item[(i)]For any polynomial $f$ we have
\begin{equation}\label{eqmain1}
\of^{(r)}-f_{\min}=O\left({1\over r^2}\right).
\end{equation}
\item[(ii)] For any (nonzero) linear polynomial $f$ we have
\begin{equation}\label{eqmain2}
\of^{(r)}-f_{\min}=\Omega\left({1\over r^2}\right).
\end{equation}
\end{itemize}
\end{theorem}

Let us say a few words about the proof technique.
For the first part (i), our analysis  relies on the following two basic steps: first, we observe that  it suffices to consider the case when  $f$ is linear (which follows using Taylor's theorem), and then we show how to reduce  to  the case of minimizing a linear univariate polynomial  over the interval $[-1,1]$, where we can rely on the analysis completed in \cite{DKL MOR2}.
  For the second part (ii),  by exploiting a connection  recently mentioned in \cite{MPSV} between the bounds (\ref{fundr}) and cubature rules, we can rely on known results for cubature rules  on the unit sphere to show tightness of the bounds.

\bigskip
{\bf Organization of the paper.}
In Section \ref{secpreli} we recall some previously known results that are most relevant to this paper. First  we give in Section \ref{secDW} a  brief recap of the approach of Doherty and Wehner \cite{DW} for analysing  bounds for polynomial optimization over the unit sphere. After that, we recall our earlier results about the quality of the bounds (\ref{fundr}) in the case of the interval $K=[-1,1]$.
Section \ref{section1} contains our main results about the convergence analysis of the bounds (\ref{fundr}) for the unit sphere: after showing in Section \ref{section1a} that the convergence rate is in $O(1/r^2)$ we prove  in Section \ref{section1b} that the analysis is tight for nonzero linear polynomials.

\section{Preliminaries}\label{secpreli}

\subsection{The approach of Doherty \& Wehner for  the sphere}\label{secDW}

Here we  briefly sketch the approach followed by Doherty and Wehner \cite{DW} for showing the convergence rate $O(1/r)$ mentioned above in (\ref{eqDW}). Their approach applies to the case when $f$ is a {\em homogeneous} polynomial,
which enables using the tensor analysis framework. A first observation made in \cite{DW} is that we may restrict to the case when $f$ has even degree, because if $f$ is homogeneous with odd degree $d$ then we have
$$\max_{x\in\oS^{n-1}} f(x) ={d^{d/2}\over (d+1)^{(d+1)/2}} \max_{(x,x_{n+1})\in\oS^n} x_{n+1}f(x).$$
So we now assume that $f$ is  homogeneous with even degree $d=2a$.

The approach in \cite{DW} in fact also permits to analyze the following  hierarchy of {lower}  bounds on $f_{\min}$:
\begin{equation}\label{eqlow1}
\uf^{(r)}:= \sup_{\lambda\in \oR}\lambda \ \text{ s.t. } \  f(x)-\lambda \in \Sigma[x]_r  + (1-\|x\|^2)\oR[x],
\end{equation}
which are the usual sums-of-squares  bounds for polynomial optimization (as introduced in \cite{Las01,PaSt03}).
Here and throughout, $\|x\|$ denotes the Euclidean norm for real vectors.
One can verify that (\ref{eqlow1}) can be reformulated as
\begin{equation}\label{eqlow2}
\begin{array}{c}
\uf^{(r)}= \displaystyle\sup_{\lambda\in\oR} \lambda\ \text{ s.t. }\  (f(x)-\lambda \|x\|^{2a})\|x\|^{2r-2a}  \in \Sigma[x]_r + (1-\|x\|^2)\oR[x]\\
= \displaystyle\sup_{\lambda\in\oR} \lambda\ \text{ s.t. } \ f(x)\|x\|^{2r-2a} -\lambda\|x\|^{2r} \in \Sigma[x]
\end{array}
\end{equation}
(see \cite{KLP05}). For any integer $r\in\oN$ we have
$$\uf^{(r)}\le f_{\min}\le \of^{(r)}.$$ The following error estimate is shown on the range $\of^{(r)}-\uf^{(r)}$ in \cite{DW}.

\begin{theorem}\cite{DW}\label{theoDW}
Assume $n\ge 3$ and $f$ is a homogeneous polynomial of degree $2a$. There exists a constant $C_{n,a}$ (depending only on $n$ and $a$)
such that, for any integer $r\ge a(2a^2+n-2)-n/2$, we have
$$\of^{(r)}-\uf^{(r)} \le {C_{n,a} \over r} (f_{\max}-f_{\min}),$$
where $f_{\max}$ is the maximum value of $f$ taken over $\oS^{n-1}$.
\end{theorem}

The starting point in the approach in \cite{DW} is reformulating the problem in terms of tensors. For this  we need the following notion of `maximally symmetric matrix'.  Given a  real symmetric matrix $M=(M_{\ui,\uj})$ indexed by sequences $\ui\in [n]^a$, $M$ is called  {\em maximally symmetric} if it is invariant under action of the permutation group $\text{Sym}(2a)$ after viewing $M$ as a  $2a$-tensor acting on $\oR^n$. This notion is  the analogue of the `moment matrix' property, when expressed in the tensor setting. To see this, for a sequence $\ui=(i_1,\ldots,i_a)\in [n]^a$, define $\alpha(\ui)=(\alpha_1,\ldots,\alpha_n)\in \oN^n$ by letting $\alpha_\ell$ denote the number of occurrences of $\ell$ within the multi-set $\{i_1,\ldots,i_a\}$ for each $\ell\in [n]$, so that $a=|\alpha|=\sum_{i=1}^n \alpha _i$. Then, the matrix $M$ is maximally symmetric if and only if each entry $M_{\ui,\uj}$ depends only on the $n$-tuple $\alpha(\ui)+\alpha(\uj)$.
Following \cite{DW} we let $\MSymna$ denote the set of maximally symmetric matrices acting on $(\oR^n)^{\otimes a}$.

It is not difficult to see that any degree $2a$ homogeneous polynomial $f$ can be represented in a unique way as
$$f(x)=(x^{\otimes a})^T Z_f x^{\otimes a},$$
where the matrix $Z_f$ is maximally symmetric.

Given an integer $r\ge a$, define the polynomial $f_r(x)=f(x)\|x\|^{2r-2a}$, thus homogeneous with degree $2r$.
The parameter  (\ref{eqlow2}) can now be reformulated as
\begin{equation}\label{eqlow3}
\uf^{(r)}= \sup\{\langle Z_{f_r},M\rangle: M\in \MSymnr,\ M\succeq 0,\ \text{Tr}(M)=1\}.
\end{equation}

The approach in \cite{DW} can  be sketched  as follows.  Let $M$ be an optimal solution to the program (\ref{eqlow3}) (which exists since the feasible region is a compact set).
Then the polynomial
$Q_M(x):= (x^{\otimes r})^T M x^{\otimes r}$
is a sum of squares since $M\succeq 0$.  After scaling, we obtain the polynomial $$h(x)=Q_M(x)/\int_{\oS^{n-1}} Q_M(x)d\sigma(x) \in \Sigma[x]_r,$$ which defines a probability density function on $\oS^{n-1}$, i.e.,  $\int_{\oS^{n-1}} h(x)d\sigma(x)=1$. In this way $h$ provides a feasible solution for the program defining the upper bound $\of^{(r)}$. This thus implies the chain of inequalities
$$\langle Z_{f_r},M\rangle =\uf^{(r)}\le f_{\min}\le \of^{(r)}\le \int_{\oS^{n-1}} f(x)h(x)d\sigma(x).$$
The  
main contribution  in \cite{DW} is their analysis for  bounding the range between the two extreme values in the above chain and showing Theorem \ref{theoDW},
 which is done by using, in particular,  Fourier analysis on the unit sphere.





Using different techniques we will show below a  rate of convergence in $O(1/r^2)$ for the upper bounds $\of^{(r)}$, thus stronger than the rate $O(1/r)$ in Theorem \ref{theoDW} above and applying to any polynomial (not necessarily homogeneous).
On the other hand, while the constant involved in Theorem \ref{theoDW}  depends only on the degree of $f$ and the dimension $n$, the constant in our result depends also on other characteristics of $f$ (its first and second order derivatives).
A key ingredient in our analysis will be to reduce to the univariate case, namely to the optimization of a linear polynomial over the interval $[-1,1]$. Thus we next recall the relevant known results that we will need in our treatment.

\subsection{Convergence analysis for  the interval $[-1,1]$}\label{seccube}

We start with recalling the  following eigenvalue reformulation for the bound (\ref{fundr}), which holds for general $K$ compact and plays a key role in the analysis for the case $K=[-1,1]$.
 For this consider the following inner product
$$(f,g)\mapsto \int_K f(x)g(x)d\mu(x) $$
on the space of polynomials on $K$ and
 let $\{b_\alpha(x):\alpha \in \oN^n\}$ denote a basis of this polynomial  space  that is orthonormal with respect to the above inner product; that is,
  $\int_K b_\alpha(x)b_{\beta}(x)d\mu(x)=\delta_{\alpha,\beta}.$
  Then the bound (\ref{fminkreform2}) can be equivalently rewritten as

 \begin{equation}\label{fminkreform3}
\of^{(r)}= \lambda_{\min}(A_f), \quad \text{ where }  A_f=\left(\int_K f(x)b_\alpha(x)b_{\beta}(x)d\mu(x)\right)_{{\alpha,\beta\in \oN^n}\atop { |\alpha|,|\beta|\le r}}
\end{equation}
(see \cite{Las11,DKL MOR1}).
Using this reformulation we could show in \cite{DKL MOR1} that the bounds (\ref{fundr}) have a convergence rate in $O(1/r^2)$ for the case of the interval $K=[-1,1]$ (and as an application also for the $n$-dimensional box $[-1,1]^n$).

This result holds for a large class of measures on $[-1,1]$, namely those which admit a weight function  $w(x)=(1-x)^a(1+x)^b$ (with $a,b>-1$) with respect to  the Lebesgue measure. The corresponding orthogonal polynomials are known as the Jacobi polynomials $P^{a,b}_d(x)$ where $d\ge 0$ is their degree.
 The  case $a=b=-1/2$ (resp., $a=b=0$) corresponds to the Chebychev polynomials (resp., the Legendre polynomials), and
 when $a=b=\lambda -1/2$, the corresponding polynomials are  the Gegenbauer polynomials $C_d^{\lambda}(x)$ where $d$ is their degree. See, e.g., \cite[Chapter 1]{DuX01} for a general reference about orthogonal polynomials.

 The key fact is that, in the case of the univariate polynomial  $f(x)=x$, the matrix $A_f$ in (\ref{fminkreform3}) has a  tri-diagonal shape, which follows from  the 3-term recurrence relationship satisfied by the orthogonal polynomials. In fact, $A_f$ coincides with the so-called  Jacobi matrix of the orthogonal polynomials  in the theory of orthogonal polynomials and its eigenvalues are given by the roots of the degree $r+1$ orthogonal polynomial (see, e.g. \cite[Chapter 1]{DuX01}). This fact is key to the following result.

\begin{theorem}\cite{DKL MOR1}\label{theocube}
Consider the measure $d\mu(x)=(1-x)^a(1+x)^bdx$ on the interval $[-1,1]$, where $a,b>-1$. For the univariate polynomial $f(x)=x$, the parameter  $\of^{(r)}$ is equal to the smallest root of the  Jacobi polynomial $P^{a,b}_{r+1}$ (with degree $r+1$). In particular,  $\of^{(r)} =-\cos\Big({\pi\over 2r+2}\Big)$ when $a=b=-1/2$. For any $a,b>-1$ we have
$$\of^{(r)} -f_{\min} = \of^{(r)}+1=\Theta\Big({1\over r^2}\Big).$$
\end{theorem}

\newcommand{\la}{\lambda}

\section{Convergence analysis for the unit sphere}\label{section1}

In this section  we analyze the quality of the bounds $\of^{(r)}$ when minimizing a polynomial $f$ over the unit sphere $\oS^{n-1}$.
In Section \ref{section1a}  we show that the range $\of^{(r)} -f_{\min}$ is in $O(1/r^2)$ and in Section \ref{section1b} we show that the analysis is tight for  linear polynomials.

\subsection{The bound $O(1/r^2)$}\label{section1a}

We first deal with the $n$-variate linear (coordinate) polynomial $f(x)=x_1$ and after that we will indicate how the general case can be reduced  to this special case.
The key idea is to get back to the analysis in Section \ref{seccube}, for the interval $[-1,1]$  with  an appropriate weight function.
We begin with  introducing some notation we need.

To simplify notation we set $d=n-1$ (which also matches the notation customary in the theory of orthogonal polynomials where $d$ usually is the number of variables). We let $\oB^d=\{x\in\oR^d: \|x\|\le 1\}$ denote the unit ball in $\oR^d$, where $\|x\|^2=\sum_{i=1}^dx_i^2$ for $x\in\oR^d$.
Given a scalar $\la>-1/2$, define the $d$-variate weight function
\begin{equation}\label{eqw}
w_{d,\la}(x)=(1-\|x\|^2)^{\la-1/2}
\end{equation}
(well-defined when $\|x\|<1$) and set
\begin{equation}\label{eqC}
C_{d,\la}:= \int_{\oB^d} w_{d,\la}(x_1,\ldots,x_d)dx_1\cdots dx_d = {\pi^{d/2}\Gamma\Big(\la+{1\over 2}\Big)\over \Gamma\Big(\la+{d+1\over 2}\Big)}
\end{equation}
so that $C_{d,\la}^{-1}w_{d,\la}(x_1,\ldots,x_d)dx_1\cdots dx_d$ is a probability measure over the unit ball $\oB^d$.
See, e.g.,  \cite[Section 2.3.2]{DuX01} or \cite[Section 11]{DaiX13}.

We will use the following simple lemma, which indicates how to integrate the $d$-variate weight function $w_{d,\la}$ along $d-1$ variables.

\begin{lemma}\label{lem0}
Fix $x_1\in [-1,1]$ and let {$d\ge 2$.} Then we have:
$$\int_{\{(x_2,\ldots,x_d):x_2^2+\ldots+x_d^2\le 1-x_1^2\}} w_{d,\la}(x_1,\ldots,x_d)dx_2\cdots dx_d = C_{d-1,\la}(1-x_1^2)^{\la +{d-2\over 2}},$$
which is thus equal to $C_{d-1,\la}w_{1,\la+(d-1)/2}(x_1)$.
\end{lemma}

\begin{proof}
Change variables and set $u_j= {x_j/ \sqrt{1-x_1^2}}$ for $2\le j\le d$. Then we have
$w_{d,\la}(x)=(1-x_1^2 -x_2^2+\ldots -x_d^2)^{\la-{1\over 2}}= (1-x_1^2)^{\la-{1\over 2}}(1-u_2^2-\ldots -u_d^2)^{\la-{1\over 2}}$ and
$dx_2\cdots dx_d= (1-x_1^2)^{d-1\over 2} du_2\cdots du_d.$
Putting things together and using relation (\ref{eqC}) we obtain the desired result.
\end{proof}

We  also need the following lemma, which relates integration over the unit sphere $\oS^d\subseteq \oR^{d+1}$ and integration over the unit ball $\oB^d\subseteq \oR^{d}$ and can be found, e.g., in \cite[Lemma 3.8.1]{DuX01} and \cite[Lemma 11.7.1]{DaiX13}.

\begin{lemma}\label{lemSB} 
Let $g$ be a $(d+1)$-variate integrable function defined on $S^d$ and {$d\ge 1$.} Then we have:
$$\int_{\oS^d} g(x)d\sigma(x)= \int_{\oB^d} \left( g(x,\sqrt{1-\|x\|^2})+ g(x,-\sqrt{1-\|x\|^2})\right) {dx_1\cdots dx_d\over \sqrt{1-\|x\|^2}}.$$
\end{lemma}

By combining these two lemmas we obtain the following result.

\begin{lemma}\label{lem1}
Let $g(x_1)$ be  a univariate polynomial and {$d\ge 1$.} Then we have:
$$\sigma_d^{-1} \int_{\oS^d}g(x_1) d\sigma(x_1,\ldots,x_{d+1})=
C_{1,\nu}^{-1} \int_{-1}^1 g(x_1) w_{1,\nu}(x_1)  dx_1,$$
where we set $\nu= {d-1\over 2}.$
\end{lemma}

\begin{proof}
Applying Lemma \ref{lemSB} to the function $x\in\oR^{d+1}\mapsto g(x_1)$ we get
\begin{equation}\label{eqid}
\sigma_d^{-1} \int_{\oS^d}g(x_1) d\sigma(x_1,\ldots,x_{d+1})=
2\sigma_d^{-1} \int_{\oB^d} g(x_1) w_{d,0}(x)dx_1\cdots dx_d.
\end{equation}
If $d=1$ then $\nu=0$ and  the right hand side  term in (\ref{eqid}) is equal to
$$2\sigma_1^{-1} \int_{-1}^1 g(x_1) w_{1,0}(x_1) dx_1= C_{1,0}^{-1}  \int_{-1}^1 g(x_1) w_{1,0}(x_1) dx_1,$$
as desired, since $2\sigma_1^{-1}C_{1,0}=1$ using $\sigma_1=2\pi$ and $C_{1,0}= \pi$ (by (\ref{eqC}) and $\Gamma(1/2)=\sqrt \pi$).
Assume now $d\ge 2$. Then the right hand side  in (\ref{eqid}) is equal to
 $$
2\sigma_d^{-1}\int_{-1}^1 g(x_1)\left( \int_{x_2^2+\ldots+x_d^2\le 1-x_1^2} w_{d,0}(x_1,\ldots,x_d) dx_2\cdots dx_d\right) dx_1$$
$$=2\sigma_d^{-1}C_{d-1,0}\int_{-1}^1g(x_1) (1-x_1^2)^{(d-2)/2} dx_1
= 2\sigma_d^{-1}C_{d-1,0}\int_{-1}^1 g(x_1) w_{1,\nu}(x_1) dx_1,$$
where we have used Lemma \ref{lem0} for the first equality.
Finally we  verify that the constant $2\sigma_d^{-1}C_{d-1,0} C_{1,\nu}$ is equal to 1:
$$2\sigma_d^{-1}C_{d-1,0}C_{1,\nu}
= 2 {\Gamma\left({d+1\over 2}\right) \over 2\pi^{d+1\over 2}}
{\pi^{d-1\over 2}\Gamma\left({1\over 2}\right)\over  \Gamma\left({d\over 2}\right)}
{\pi^{1\over 2} \Gamma\left({d\over 2}\right)\over \Gamma\left({d+1\over 2}\right)}=1$$
(using relations (\ref{eqsigma}) and (\ref{eqC})), and thus we  arrive at the desired identity.
\end{proof}

We can now complete the  convergence analysis for the minimization of $x_1$ on the unit sphere.

\begin{lemma}\label{lemx1}
For the minimization of the polynomial $f(x)=x_1$ over $\oS^d$ with $d\ge 1$,  the order $r$ upper bound (\ref{fundr}) satisfies
$$\of^{(r)}= -1 +O\left({1\over r^2}\right).$$
\end{lemma}

\begin{proof}
Let $h(x_1)$ be an optimal univariate sum-of-squares polynomial of degree $2r$ for the order $r$ upper bound corresponding to  the minimization of $x_1$ over $[-1,1]$, when using  as reference measure on $[-1,1]$ the measure with  weight function $w_{1,\nu}(x_1)C_{1,\nu}^{-1}$ and $\nu=(d-1)/2$ (thus $\nu >-1$).  Applying Lemma \ref{lem1}  to the univariate polynomials $h(x_1)$ and $x_1h(x_1)$, we obtain
$$\sigma_d^{-1}\int_{\oS^d} h(x_1) d\sigma(x)=C_{1,\nu}^{-1}\int_{-1}^1 h(x_1)w_{1,\nu}(x_1)dx_1 =1$$ and
$$\of^{(r)}\le \sigma_d^{-1}\int_{\oS^d}x_1h(x_1) d\sigma(x)= C_{1,\nu}^{-1}\int_{-1}^1 x_1h(x_1) w_{1,\nu}(x_1)dx_1.$$
Since the function  $x_1$ has the same global minimum $-1$ over $[-1,1]$ and over the sphere $\oS^d$, 
we can apply Theorem \ref{theocube} to conclude that $$\of^{(r)}+1\le 1+ C_{1,\nu}^{-1}\int_{-1}^1 x_1h(x_1) w_{1,\nu}(x_1)dx_1 =O\Big({1\over r^2}\Big).$$
\end{proof}

We now indicate how the analysis for an arbitrary polynomial $f$ reduces to the case of the linear coordinate polynomial  $x_1$.
To see this, suppose $a\in \oS^{n-1}$ is a global minimizer of $f$ over $\oS^{n-1}$. Then, using Taylor's theorem, we can upper estimate $f$ as follows:
$$\begin{array}{llll}
f(x) & \le & f(a) +\nabla f(a)^T (x-a)+{1\over 2}C_f\|x-a\|^2  &  \forall  x\in \oS^{n-1} \\
& = & f(a) +\nabla f(a)^T (x-a)+C_f (1-a^Tx)=: g(x)  &\forall  x\in \oS^{n-1},
\end{array}
$$
setting $C_f=\max_{x\in \oS^{n-1}} \|\nabla^2f(x)\|_2$. Note that the upper estimate $g(x)$ is a linear polynomial, which has the same minimum value as $f(x)$ on
$\oS^{n-1}$, namely $f(a)=f_{\min}=g_{\min}$. From this it follows that
$\of^{(r)}-f_{\min}\le \overline g^{(r)}-g_{\min}$ and thus we may restrict to analyzing the bounds for a linear polynomial.

Next, assume $f$ is a linear polynomial, of the form $f(x)=c^Tx$ with (up to scaling)   $\|c\|=1$.  We can then apply a change of variables to bring $f(x)$ into the form $x_1$. Namely, let $U$ be an orthogonal $n\times n$ matrix such that $Uc=e_1$. Then the polynomial $g(x):=f(U^Tx)= x_1$ has the desired form and it has the same minimum value $-1$ over $\oS^{n-1}$ as $f(x)$. As the sphere is invariant under any orthogonal transformation  it follows that $\of^{(r)}=\overline g^{(r)}= -1+O(1/r^2)$ (applying Lemma \ref{lemx1} to $g(x)=x_1$).
Summarizing, we have shown the following.

\begin{theorem}\label{theo1}
For the minimization of any polynomial $f(x)$ over $\oS^{n-1}$ with $n\ge 2$,  the order $r$ upper bound (\ref{fundr}) satisfies
$$\of^{(r)}-f_{\min}= O\left({1\over r^2}\right).$$
\end{theorem}

Note the difference to Theorem \ref{theoDW} where the constant  depends only on the degree of $f$ and the number $n$ of variables; here the constant in $O(1/r^2)$ does also depend on the polynomial $f$,
namely it depends on  the norm of $\nabla f(a)$ at a global minimizer $a$ of $f$ in $\oS^{n-1}$ and on $C_f=\max_{x\in \oS^{n-1}}\|\nabla^2 f(x)\|_2$.

\subsection{The analysis is tight for linear polynomials} \label{section1b}

In this section we show --- through an example --- that the convergence rate cannot be better than $\Omega\left(1/r^2\right)$.
The example is simply minimizing $x_1$ over the sphere $\oS^{n-1}$. 
The key tool we use is a link between the bounds  $\of^{(r)}$ and  properties of some known cubature rules on the unit sphere. This connection, recently mentioned in \cite{MPSV},
 holds for any
compact set $K$. It goes as follows.

Suppose the points $x^{(1)},\ldots,x^{(N)}\in K$ and the weights $w_1,\ldots,w_N>0$ provide a (positive) cubature rule for $K$ for a given  measure $\mu$, which is exact up to degree $d+2r$, that is,
$$\int_K g(x) d\mu(x)=\sum_{i=1}^N w_i g(x^{(i)})$$
for all polynomials $g$ with degree at most $d+2r$.
Then, for any polynomial $f$ with degree at most $d$, we have
\begin{equation}\label{eqcubature}
\of^{(r)} \ge \min_{i=1}^N f(x^{(i)}).
\end{equation}
The argument is  simple: if  $h\in \Sigma[x]_r$ is an optimal sum-of-squares density for the parameter $\of^{(r)}$, then we have
$$1=\int_K h(x)d\mu(x)= \sum_{i=1}^N w_i h(x^{(i)}),$$
$$\of^{(r)}=\int_K f(x)h(x)d\mu(x)=\sum_{i=1}^N w_i f(x^{(i)})h(x^{(i)}) \ge \min_i f(x^{(i)}).$$

As a warm-up we consider the case $n=2$, where  we can use the  cubature rule in Theorem \ref{theocircle} below for the unit circle.
We use spherical coordinates $(x_1,x_2)=(\cos\theta,\sin\theta)$ to express a polynomial $f$ in $x_1,x_2$ 
as a polynomial $g$ in $\cos \theta,\sin\theta$.

\begin{theorem}{\rm \cite[Proposition 6.5.1]{DaiX13}}\label{theocircle}
For each $d \in\oN$, the cubature formula
\[
\frac{1}{2\pi}\int_0^{2\pi} g(\theta) d\theta = \frac{1}{d} \sum_{j=0}^{d-1}g\left(  \frac{2\pi j}{d}\right)
\]
is exact for all  $g \in \mbox{span}\{1,\cos \theta,\sin \theta,\ldots,\cos (d\theta), \sin (d\theta)\}$, i.e.\ for all polynomials
of degree at most $d$, restricted to the unit circle.
\end{theorem}

Using this cubature rule on $\oS^1$ we can lower bound the parameters $\of^{(r)}$ for the minimization of $f(x)=x_1$ over $\oS^1$. Namely, by setting  $x_1=\cos \theta $, we derive directly from the above theorem combined with relation (\ref{eqcubature}) that
$$\of^{(r)} \ge \min_{0\le j\le 2r} \cos \Big({2\pi j\over 2r+1}\Big)= \cos \Big( {2\pi r\over 2r+1}\Big) = -1 + \Omega\Big({1\over r^2}\Big).$$


This  reasoning  extends to any dimension $n\ge 2$, by using product-type cubature formulas on the sphere $\oS^{n-1}$.
In particular we will use the cubature rule described in \cite[Theorem 6.2.3]{DaiX13}, see Theorem \ref{theocubature} below.

We will need the generalized spherical coordinates given by
\begin{equation}
\label{spherical coordinates}
\left.
\begin{array}{rcl}
x_1 &=& r\sin \theta_{n-1}\cdots \sin \theta_3 \sin \theta_2\sin \theta_1 \\
x_2 &=& r\sin \theta_{n-1}\cdots \sin \theta_3\sin \theta_2 \cos \theta_1 \\
x_3 &=& r\sin \theta_{n-1}\cdots \sin \theta_3  \cos \theta_2 \\
 &\vdots & \\
x_{n} &=& r\cos \theta_{n-1},
\end{array}
\right\}
\end{equation}
where $r \ge 0$ ($r=1$ on $\oS^{n-1}$), $0 \le \theta_1 \le 2 \pi$, and $0 \le \theta_i \le  \pi$ ($i=2,\ldots,n-1$).

To define the nodes of the cubature rule on $\oS^{n-1}$ we need the Gegenbauer polynomials  $C^\lambda_d(x)$, where $\lambda > -1/2$. Recall that these are the orthogonal polynomials  with respect to the weight function
\[
w_{1,\lambda}(x) = (1-x^2)^{\lambda- 1/2} \quad x \in (-1,1)
\]
on $[-1,1]$. We will not need the explicit expressions for the polynomials $C^\lambda_d(x)$, we only need the following information about their extremal roots, shown in \cite{DKL MOR1} (for general Jacobi polynomials,  using results of \cite{DN10,DJ12}). It is well known that each $C_d^{\lambda}(x)$ has $d$ distinct roots, lying in $(-1,1)$.

\begin{theorem}\label{theoroot}
Denote the roots of the polynomial $C^\lambda_d(x)$ by $t^{(\lambda)}_{1,d} < \ldots < t^{(\lambda)}_{d,d}$.
Then, $t^{(\lambda)}_{1,d}  +1 =  \Theta(1/d^2)$.
\end{theorem}

The cubature rule we will use may now be stated.

\begin{theorem}{\rm \cite[Theorem 6.2.3]{DaiX13}}\label{theocubature}
Let $f:\oS^{n-1} \rightarrow \mathbb{R}$ be a polynomial of degree at most $2d-1$, and let
\[
g(\theta_1,\ldots,\theta_{n-1}) := f(x_1,\ldots,x_n),
\]
be the expression of $f$ in the generalized spherical coordinates \eqref{spherical coordinates}. Then
\begin{equation}
\label{product cubature rule}
\int_{\oS^{n-1}} f(x)d\sigma(x) = \frac{\pi}{d} \sum_{k=0}^{2d-1}\sum_{j_2 = 1}^d\cdots \sum_{j_{n-1} = 1}^d\prod_{i=2}^{n-1}
{\mu^{((i-1)/2)}_{i,d}}
g\left(\frac{\pi k}{d},\theta^{(1/2)}_{j_2,d},\ldots,\theta^{(({n}-2)/2)}_{j_{n-1},d}\right),
\end{equation}
where {$\cos \left(\theta^{(\lambda)}_{j,d}\right) :=  t^{(\lambda)}_{j,d}$} {and the parameters $\mu^{((i-1)/2)}_{i,d}$ are positive scalars as in relation (6.2.3) of \cite{DaiX13}.}
\end{theorem}

We can now show the tightness of the convergence rate $\Omega(1/r^2)$ for the minimization of a coordinate  polynomial on $\oS^{n-1}$.

\begin{theorem}\label{theocubn}
Consider the problem of minimizing the coordinate polynomial $x_n$ on the unit sphere $\oS^{n-1}$ with $n\ge 2$.
The convergence rate for the parameters (\ref{fundr})  satisfies
\[
\of^{(d)}- f_{\min} =\of^{(d)}+1= \Omega\left( \frac{1}{d^2}\right).
\]
\end{theorem}

\begin{proof}
We have $f(x_1,\ldots,x_n) = x_n$, so that $ g(\theta_1,\ldots,\theta_{n-1}) = \cos \theta_{n-1}$.
Using (\ref{eqcubature}) we obtain that
$$\of^{(r)}\ge \min_{1\le j\le d} \cos \theta^{((n-2)/2)}_{j,d} =\min_{1\le j\le d} t^{((n-2)/2)}_{j,d} = t^{((n-2)/2)}_{1,d}= -1 + \Omega\Big({1\over d^2}\Big),$$
where we  use 
 the fact that $t^{(\lambda)}_{1,d}  +1 =  \Theta(1/d^2)$ (Theorem \ref{theoroot}).

\end{proof}

\section{Implications for the generalized problem of moments}

In this section, we describe the implications of our results for the generalized problem of moments (GPM), defined as follows for a compact
set $K \subset \mathbb{R}^n$.
\begin{equation}
\label{prob:GPM}
val := \inf_{\nu \in \mathcal{M}(K)_+} \left\{ \int_K f_0(x)d\nu(x) \; : \; \int_K f_i(x)d\nu(x) = b_i \quad \forall i \in [m]\right\},
\end{equation}
where
\begin{itemize}
\item
the functions $f_i$ $(i=0,\ldots,m)$ are continuous on $K$;
\item
$\mathcal{M}(K)_+$ denotes the convex cone of probability measures supported on the set $K$;
\item
the scalars $b_i\in\oR$ ($i \in [m]$) are given.
\end{itemize}
As before, we are interested in the special case where $K = \oS^{n-1}$. This special case is already of independent
interest, since it contains the problem of finding cubature schemes for numerical integration on the sphere, see e.g.\ \cite{DKL survey} and the references therein.
Our main result in Theorem \ref{theo1} has the following implication for the GPM on the sphere, as a corollary of the following result in \cite{DK-P-K dist robust opt} (which applies to any compact $K$, see also
\cite{DKL survey} for a sketch of the proof in the setting described here).

\begin{theorem}[De Klerk-Postek-Kuhn \cite{DK-P-K dist robust opt}]
\label{th:moments sos}
Assume that $f_0,\ldots,f_m$ are polynomials, $K$ is compact, $\mu$ is a Borel measure supported on $K$, and  the GPM \eqref{prob:GPM} has an optimal solution.
Given $r \in \mathbb{N}$, define  the parameter
$$\Delta(r)= \min_{h\in \Sigma_r} \max_{i\in\{0,1,\ldots,m\}} \big | \int_K f_i(x)h(x)d\mu(x) -b_i\big |,$$
setting $b_0=val$.
If,  for any polynomial $f$, we have
\[
\of^{(r)}_K - f_{\min} = O(\varepsilon(r)),
\]
where $\lim_{r \rightarrow \infty} \varepsilon(r) = 0$, then the parameters $\Delta(r)$ satisfy: 
$\Delta(r) = O(\sqrt{\varepsilon(r)})$.
\end{theorem}

As a consequence of our main result in Theorem \ref{theo1}, combined with Theorem \ref{th:moments sos}, we immediately obtain the following corollary.

\begin{corollary}
\label{cor:moments sos}
Assume that $f_0,\ldots,f_m$ are polynomials, $K = \oS^{n-1}$, and  the GPM \eqref{prob:GPM} has an optimal solution.
Then, for any integer $r\in\oN$, there is an $h_r \in \Sigma_{r}$ such that
\[
\left| \int_{\oS^{n-1}} f_0(x) h_r(x) d\sigma(x) - val \right| = O(1/r), \; \left| \int_{\oS^{n-1}} f_i(x) h_r(x) d\sigma(x) - b_i \right| = O(1/r) \quad \forall i \in [m].
\]
\end{corollary}

Minimization of a rational function on $K$ is a special case of the GPM where we may prove a better rate of convergence.
In particular, we now consider the global optimization problem:
\begin{equation}\label{eq:rational global opt}
val =  \min_{x \in K} \frac{p(x)}{q(x)},
\end{equation}
where $p,q$ are polynomials such that $q(x) > 0$ $\forall$ $x \in K$, and $K \subseteq \mathbb{R}^n$ is  compact.

It is well-known that one may reformulate this problem as the GPM with $m=1$ and $f_0 = p$,  $f_1 = q$, and $b_1 =1$, i.e.:
\[
val = \min_{\nu \in \mathcal{M}(K)_+} \left\{ \int_K p(x)d\nu(x) \; : \; \int_K q(x)d\nu(x) = 1\right\}.
\]

Analogously to \eqref{fundr}, we now define the hierarchy of upper bounds on $val$ as follows:
\begin{equation}\label{rational hierarchy}
 \overline{p/q}^{(r)}_K:=\min_{h\in\Sigma[x]_r}\int_{{K}}p(x)h(x)d\mu(x) \ \ \mbox{s.t. $\int_{{K}}q(x)h(x)d\mu(x)=1$,}
\end{equation}
where $\mu$ is a Borel measure supported on $K$.
\begin{theorem}
\label{th:moments sos}
Consider the rational optimization problem \eqref{eq:rational global opt}.
If, for any polynomial $f$, it holds that
\[
\of^{(r)}_K - f_{\min} = O(\varepsilon(r))
\]
where $\lim_{r \rightarrow \infty} \varepsilon(r) = 0$, then one also has
$\overline{p/q}^{(r)}_K - val = O(\varepsilon(r))$. In particular, if $K = \oS^{n-1}$, then
$\overline{p/q}^{(r)}_K - val = O(1/r^2)$.
\end{theorem}
\proof
Consider the polynomial
\[
f(x) = p(x) - val \cdot q(x).
\]
Then $f(x) \ge 0$ for all $ x \in K$, and $f_{\min,K} = 0$, with global minimizer given by the minimizer of problem \eqref{eq:rational global opt}.

Now, for given $r \in \mathbb{{N}}$,  let $h \in \Sigma_r$ be such that $\of^{(r)}_K = \int_K f(x)h(x)d\mu(x)$, and $\int_K h(x)d\mu(x) = 1$, where $\mu$ is the reference measure for $K$.
Setting
\[
h^* = \frac{1}{\int_K h(x)q(x)d\mu(x)}h,
\]
one has $h^* \in \Sigma_r$ and $\int_K h^*(x)q(x)d\mu(x) = 1$. Thus $h^*$ is feasible for problem \eqref{rational hierarchy}. Moreover, by construction,
\begin{eqnarray*}
 \int_K p(x)h^*(x)d\mu(x) - val &=& \frac{ \of^{(r)}_K}{\int_K h(x)q(x)d\mu(x)} \\
   &\le & \frac{ \of^{(r)}_K}{\min_{x \in K} q(x)} = O(\varepsilon(r)).
\end{eqnarray*}
The final result for the special case $K = \oS^{n-1}$ and $\mu = \sigma$ (surface measure) now follows from our main result in Theorem \ref{theo1}. \qed

\section{Concluding remarks}

In this paper we have improved on the $O(1/r)$ convergence result
of Doherty and Wehner \cite{DW} for the Lasserre hierarchy of upper bounds  \eqref{fundr} for (homogeneous) polynomial optimization on the sphere.
Having said that,  Doherty and Wehner also showed that the hierarchy of \emph{lower} bounds \eqref{eqlow1} of Lasserre satisfies the same rate of convergence, due to Theorem~\ref{theoDW}.
In view of the fact that we could show the improved $O(1/r^2)$ rate for the upper bounds, and the fact that the lower bounds hierarchy empirically converges much faster in practice,
one would expect that the lower bounds \eqref{eqlow1} also converge at a rate no worse than $O(1/r^2)$. However, our analysis does not allow us to analyse the convergence of the lower bound hierarchy, and this remains
an interesting open problem.

Another open problem is the exact rate of convergence of the bounds in Theorem \ref{th:moments sos} for the generalized problem of moments (GPM).
In our analysis of the GPM on the sphere in Corollary \ref{cor:moments sos}, we could only obtain $O(1/r)$ convergence, which is a square root worse than the special cases for polynomial and rational function minimization.
We do not know at the moment if this is a weakness of the analysis or inherent to the GPM.


{Note that if we pick another reference measure $d\mu(x)=q(x)d\sigma(x)$, where $q$ is strictly positive on the sphere, then  the convergences rates with respect to both measures $\sigma$ and $\mu$ have the same behaviour (up to multiplicative constant). It would be interesting to understand the convergence rate for more general reference measures.}

\subsubsection*{Acknowledgement}
{This work has been supported by European Union's Horizon 2020 research and innovation programme under the Marie Sk\l odowska-Curie grant agreement 813211  (POEMA).}

 \end{document}